\documentclass[12pt]{article}
\usepackage{amssymb,amsmath,amsthm,amsfonts}
\usepackage{mathrsfs}
\usepackage{graphicx}
\usepackage{epsfig}
\usepackage{tikz}
\usepackage{mathtools}

 \newcommand{\RNum}[1]{\uppercase\expandafter{\romannumeral #1\relax}}

\theoremstyle{plain}
\newtheorem{theorem}{Theorem}[section]
\newtheorem{lemma}{Lemma}[section]

\newtheorem{corollary}{Corollary}[section]

\theoremstyle{definition}

\numberwithin{equation}{section}

\begin{document}

\title{\large\bf Asymptotic behavior of a three-dimensional haptotactic cross-diffusion system modeling oncolytic virotherapy}

\author{
{\rm  Yifu Wang$^{1,*}$, Chi Xu$^{1}$
}\\[0.2cm]
{\it\small \rm School of Mathematics and Statistics,  Beijing Institute of Technology}\\
{\it\small \rm Beijing 100081, P.R. China$^1$}
}
\date{}

\maketitle
\par
\date{}
\maketitle
\vspace{0.1cm}
\noindent

\begin{abstract}
This paper deals with an initial-boundary value problem for a doubly haptotactic cross-diffusion system arising from the oncolytic virotherapy
\begin{equation*}
\left\{
\begin{array}{lll}
u_t=\Delta u-\nabla \cdot(u\nabla v)+\mu u(1-u)-uz,\\
v_t=-(u+w)v,\\
w_t=\Delta w-\nabla \cdot(w\nabla v)-w+uz,\\
z_t=D_z\Delta z-z-uz+\beta w,	
\end{array}
\right.
\end{equation*}
in a smoothly bounded domain $\Omega\subset \mathbb{R}^3$ with $\beta>0$,~$\mu>0$ and $D_z>0$.  Based on a self-map argument,
 it is shown that under the assumption  $\beta \max \{1,\|u_0\|_{L^{\infty}(\Omega)}\}< 1+
(1+\frac1{\min_{x\in \Omega}u_0(x)})^{-1}$, this problem   possesses a uniquely
determined global classical solution $(u,v,w,z)$ for certain type of small data $(u_0,v_0,w_0,z_0)$. Moreover, $(u,v,w,z)$ is globally bounded and exponentially stabilizes towards its spatially homogeneous
equilibrium
 $(1,0,0,0)$ as $t\rightarrow \infty$.
\end{abstract}

\section{Introduction}
Oncolytic virus (OV) is either a natural or genetic-engineered virus which can  specifically  infect cancer cells,
enlarge the quantities through replication inside the cancer cells and eventually  lyse them, while ideally leave normal
cells unharmed\cite{FIT-CS,GPKLK-TT}.
Accordingly, comparing with the traditional treatments  towards cancer disease,
 oncolytic virus therapy has recently been recognized as a promising alternatives
and nowadays has been used in current clinical trials\cite{BKW-SC,CSFL-SC,MGHZ,RPB-NB}.
~However,~such partially success in some clinical trials still reveals the limitation in implementation. In fact, the efficacy of the so-called virotherapy will be limited by a plenty of factors, like the deposits of extracellular matrix (ECM), immune system, or the circulating
antibodies\cite{GK-CCR,NG-CR,VH-B,WLW-V}. Therefore, to better understand the underlying mechanisms that limit the efficacy of clinical treatment, Alzahrani et al.\cite{AET-MB} introduced a mathematical model to simulate the interaction between the uninfected cancer cells $u$,
infected cancer cells $w$, extracellular matrix $v$ and oncolytic virus $z$, which is formulated by
\begin{equation}\label{1.1}
\left\{\begin{array}{ll}
  u_t=D_u\Delta  u-\xi_u\nabla\cdot(u\nabla v)+\mu_u u(1-u)-\rho_u uz,\quad &x\in\Omega,~t>0\\
 v_t=- (\alpha_u u+\alpha_w w)v+\mu_v v(1-v),\quad &x\in\Omega,~t>0\\
w_t=D_w\Delta  w-\xi_w\nabla\cdot(w\nabla v)- \delta_w w+\rho_w uz,\quad &x\in\Omega,~t>0\\
z_t=D_z\Delta z-\delta_z z-\rho_z uz+\beta w,\quad &x\in\Omega,~t>0
 \end{array}\right.
\end{equation}
in a smoothly bounded domain $\Omega\subset\mathbb{R}^N$. The parameters $D_u$, $D_w,D_z$, $\xi_u$, $\xi_w,\alpha_u$, $\alpha_w,\mu_u,\delta_w,\delta_z,\beta$ are positive while $\mu_v,\rho_u,\rho_w,\rho_z$ are nonnegative.
 The underlying modeling hypothesize are that apart from its random diffusion,~both type of cancer cells will be attracted simultaneously due to some macromolecules bounded in the extracellular matrix, and that uninfected tumor cells, except perhaps proliferating logically, are converted into a state of infection in contact with
virus particles, and infected cells die due to cytolysis.~In addition,~the virus particles will adhere on the surface of cancer cells and then the newly infectious virus will be released from its inside of tumor with rate $\beta>0$;~beyond this, both of uninfected and infected cancer cells can degrade the static ECM due to the matrix degrading enzymes and the ECM is possibly  re-established according to a logistic law.

Due to the relevance with several  biological contexts, inter alia the cancer invasion or angiogenesis\cite{ACNS-JTM,AC-BMB},
 the haptotactic mechanism has received considerable attention in current
 literatures\cite{DLiu,Jin,JinTian,LL-N,LM-M3AS,LM-NA,PW3MAS,PW-M3AS,LW-JDE,SSW,TW-SIMA,WW-SIMA,ZK-JDE,ZSU-ZAMP}.
It is noted that as the striking feature thereof, two simultaneous haptotactic cross-diffusion terms distinguishes \eqref{1.1} from the most of the studied chemotaxis-haptotaxis system\cite{Jin,JinTian,LL-N,PW3MAS,TW-SIMA,ZK-JDE} and haptotaxis
system\cite{DLiu,LM-M3AS,LM-NA,PW-M3AS,WW-SIMA,ZSU-ZAMP}; apart from that, the respective mechanism in \eqref{1.1} is  potentially enhanced through
 the zero-order nonlinear production term $+\rho_w uz$  even in the two-dimensional case, which brings significant challenges to rigorous analysis
 even at the level of  elementary  solution theory.
   This reflects the situation that the comprehensive result on \eqref{1.1} seems to be rather thin, particularly in the physically relevant three-dimensional setting,
 although the qualitative information has been successfully achieved  for some simple
 variants of \eqref{1.1}\cite{C-JMAA,RenWei,LW-JDE,TW-EJAM,Taozhou,WWL}. 
For example, the issue of the global solvability of system \eqref{1.1} was addressed in\cite{TW-JDE2020} in the spatially
  two-dimensional framework,  while it is still unaddressed  in the corresponding 
 three-dimensional situation. 
     Recently,~based on the previous result of global solvability,~the $L^{\infty}$-exponential convergence property towards its equilibrium has been achieved in\cite{WX-preprint}.
On the other hand, the existing literatures in series gradually indicate that the virus production rate $\beta$ relative to the death rate of infected cancer cells is critical for the large-time behavior of the corresponding solution at least in some simple variant of \eqref{1.1} inter alia through neglecting the haptotactic movement of infected cancer cells and the renewed effect of ECM\cite{TW-NA,TW-PRSE,TW-DCDS,TW-EJAM}.  Specially, for the  cross-diffusion system  with single haptotaxis:
\begin{equation}\label{1.2}
\left\{
\begin{array}{ll}
u_t=\Delta u-\nabla \cdot(u\nabla v)+\mu u(1-u)-\rho uz,\quad &x\in\Omega,~t>0,\\
v_t=-(u+w)v,&x\in\Omega,~t>0,\\
w_t=\Delta w-w+uz,&x\in\Omega,~t>0,\\
z_t=D_z\Delta -z-uz+\beta w,	&x\in\Omega,~t>0,
\end{array}
\right.	
\end{equation}
in  a smooth bounded domain  $\Omega\subset\mathbb{R}^2$,  it is proved in\cite{TW-DCDS} that if $\beta>1$ and the initial data $\overline{u_0}>\frac{1}{\beta-1}$,~the corresponding solution to \eqref{1.2}  with $\mu=\rho=0$ will be blow-up in infinite time,  while the solution will be bounded if $\overline{u_0}<\frac{1}{(\beta-1)_{+}}$ and $v_0\equiv 0$.
However, the solution component $u$ will possess a positive lower bound when $0<\beta<1, \rho>0$\cite{TW-EJAM}. ~Futhermore,~based on the outcome of \cite{TW-EJAM},~the  asymptotic behavior of  \eqref{1.2} is established in\cite{TW-NA} if $0<\beta<1$.~It should be mentioned that in the case of $\rho>0$ and $\mu>0$, the convergence property  is also discussed in\cite{C-JMAA}.
Particularly, it is  shown in\cite{TW-PRSE} that for any prescribed level $\gamma\in(0,\frac{1}{(\beta-1)_{+}})$,~the corresponding solution of \eqref{1.2} with $\mu=0$,~$\rho\geq 0$ will  tend to the constant equilibrium $(u_{\infty},0,0,0)$ with some $u_{\infty}>0$ whenever the initial
 deviation from $(\gamma,0,0,0)$ is  suitably small,
  which provided a complement to the result of \cite{TW-DCDS}. It is observed that for two-dimensional haptotactic cross-diffusion model
  \eqref{1.2},  the result on the asymptotic behavior seems to be rather comprehensive when $0<\beta<1$,
  while it becomes rudimental  
  in the case of $\beta>1$\cite{TW-DCDS,TW-PRSE}, let alone that for the doubly haptotactic cross-diffusion system in the
  three-dimensional scenario.  Accordingly, the purpose of this work is to investigate the dynamical features involving the doubly haptotactic processes  in the
  three dimensional setting without the constraint of $0<\beta<1$.
More precisely, we shall consider a simplified version of \eqref{1.1} by neglecting the renewal of extracellular matrix,~which is given as follows				 \begin{equation}\label{1.3}
 \left\{\begin{array}{ll}
  u_t=\Delta  u-\nabla\cdot(u\nabla v)+\mu u(1-u)-uz,\quad &x\in\Omega,~t>0,\\
 v_t=- (u+w)v,&x\in\Omega,~t>0,\\
w_t=\Delta  w-\nabla\cdot(w\nabla v)- w+uz,&x\in\Omega,~t>0,\\
z_t=D_z\Delta z-z-uz+\beta w,&x\in\Omega,~t>0,\\
(\nabla u-u\nabla v)\cdot\nu=(\nabla w-w\nabla v)\cdot \nu=\nabla z\cdot\nu=0,&x\in\partial \Omega,~t>0,\\
u(x,0)=u_0,~v(x,0)=v_0,~w(x,0)=w_0,~z(x,0)=z_0,~&x\in\Omega,
 \end{array}\right.
\end{equation}
where $\Omega\subset \mathbb{R}^3$  is a bounded domain with a smooth	 boundary and $\nu$ is a outward unit normal vector of $\Omega$.~Moreover,~we henceforth assume that the initial data satisfies
\begin{equation}\label{1.4}
\left\{
\begin{array}{ll}
\displaystyle{u_0,w_0,z_0~ \hbox{and}\, v_0 ~\hbox{are  nonnegative functions from}~C^{2+\vartheta}(\bar{\Omega})~
\hbox{for some}~\vartheta\in (0,1),} \\
\displaystyle{\mbox{with}~u_0\not\equiv0,~w_0\not\equiv0,~z_0\not\equiv0,~v_0\not\equiv0~
~~\mbox{on}~~\partial\Omega.} \\
\end{array}
\right.
\end{equation}

With regard to the qualitative information of doubly haptotactic cross-diffusion systems in the
  three-dimensional setting, our  subsequence analysis reveals that even in the case of $\beta>1$,
  the uninfected cancer cells $u$ in  \eqref{1.3} is uniformly persistent for certain type of small data $(u_0,v_0,w_0,z_0)$;
in particular, whenever  $u_0(x)\leq 1$,
the  virus production rate
making possible for the  decay of uninfected cancer cells  is increasing with respect to $\displaystyle\min_{x\in \Omega}u_0(x)$, which can be stated as follows
\begin{theorem}\label{T1.1}
Let $\Omega\subset\mathbb{R}^3$ be a bounded domain with smooth boundary, $D_z>0,\mu>0$ and  $\beta \max \{1,\|u_0\|_{L^{\infty}(\Omega)}\}< 1+
(1+\frac1{\min_{x\in \Omega}u_0(x)})^{-1}$. Then  there exists $\varepsilon=\varepsilon(\beta,\mu,u_0)>0$ with the property such that whenever the initial data $(u_0,v_0,w_0,z_0)$ fulfills \eqref{1.4},
\begin{equation}\label{v1.5}
\|v_0 \|_{L^{\infty}(\Omega)}<\varepsilon,	
\end{equation}
and
\begin{equation}\label{w1.6}
\|w_0\|_{L^{\infty}(\Omega)}<\varepsilon
\end{equation}
as well as
\begin{equation}\label{z1.7}
\|z_0\|_{L^{\infty}(\Omega)}<\varepsilon,	
\end{equation}
the problem \eqref{1.3} has a unique non-negative global classical solution $(u,v,w,z)$  which is bounded in the sense that 
$$\displaystyle\sup_{t>0}\left\{\|u(\cdot, t)\|_{L^\infty(\Omega)}+\|w(\cdot, t)\|_{L^\infty(\Omega)}+\|z(\cdot, t)\|_{L^\infty(\Omega)}
+\|v(\cdot,t)\|_{W^{1,5}(\Omega)}\right\} < \infty,
$$
and moreover
\begin{equation}\label{1.8}
\|u(\cdot,t)-1\|_{L^{\infty}(\Omega)}\rightarrow 0\quad\hbox{as $t\rightarrow \infty$}	
\end{equation}
and
\begin{equation}\label{1.9}
\|v(\cdot,t)\|_{W^{1,4}(\Omega)}\rightarrow 0\quad\hbox{as $t\rightarrow \infty$}	
\end{equation}
and
\begin{equation}\label{1.10}
\|w(\cdot,t)\|_{L^{\infty}(\Omega)}\rightarrow 0\quad\hbox{as $t\rightarrow \infty$}	
\end{equation}
as well as
\begin{equation}\label{1.11}
\|z(\cdot,t)\|_{L^{\infty}(\Omega)}\rightarrow 0\quad\hbox{as $t\rightarrow \infty$}.	
\end{equation}
\end{theorem}

Our approach is based on a self-map argument, which presupposes a certain decay property on $z$ within an appropriate time interval,
and then proves that it is actually valid in the whole interval $[0,T_{max})$ with the help of some necessary a priori estimates,
  inter alia the lower pointwise bound of $a=ue^{-v}$, which results from
a parabolic comparison argument on the basis of an absorptive parabolic inequality fulfilled
by $a$ (see \eqref{3.8}). Indeed, due to the lower pointwise bound of $a$,
one can arrive at the exponential decay of $v$ (see Lemma \ref{lemma3.2}).
  As a consequence thereof, one can achieved a suitable upper bounds
for $b=we^{-v}$ as well as $a$ by  adequately exploiting the corresponding cooperative parabolic system for $(a,b)$ (see Lemma \ref{lemma3.3}),
which is somewhat different from that  in\cite{TW-PRSE}. Note that 
the assumption in Theorem \ref{T1.1} makes it possible to construct the desired super-solutions thereof.  Thereafter  
 we shall be able to prove that $z$ actually decays exponentially (see Lemma \ref{lemma3.4}) and thereby complete the self-map type reasoning.
 Furthermore, based on the $L^\infty(\Omega)$ of $(a,v,b,z)$ achieved in previous section, we prove that $(a,v,b,z)$ is  actually global, that is $T_{max}=\infty$. To this end, we need to estimate   $\nabla v$ with respect to the norm in $L^5(\Omega)$ according to extensibility criteria \eqref{2.2}.
  With respect to this, we turn to estimate the coupled quantity
  $$\int_{\Omega}|\nabla a|^2+\int_{\Omega}|\nabla b|^2+\eta\int_{\Omega}|\nabla v|^4$$
  with some $\eta>0$  through the testing processes (Lemma \ref{lemma4.1}). In Section 5,  \eqref{1.8}   is shown by means of the bootstrap method.
  Indeed,  as the first step toward this we show the convergence  of integral $\int^{\infty}_0\int_{\Omega}|\nabla v|^2$  (Lemma \ref{lemma5.2}).
  As the consequence thereof, we will then successively obtain the exponential decay property for $\|\nabla v(\cdot,t)\|_{L^2(\Omega)}$,
$\|u(\cdot,t)-1\|_{L^p(\Omega)}$ for any $p<6$, $\|\nabla v(\cdot,t)\|_{L^4(\Omega)}$ and $\|u(\cdot,t)-1\|_{L^\infty(\Omega)}$,
  thanks to the explicit expression  of $\nabla v$  and the  Gagliardo--Nirenberg type inequality.

\section{Preliminaries}
For the convenience of subsequent analysis, we introduce the variable transformation\cite{FFH,LW-JDE,TW-JDE,TW-SIMA}
$$
a=ue^{-v},\quad b=we^{-v},
$$
upon which system \eqref{1.3}  is converted to the following system
\begin{equation}\label{2.1}
\left\{\begin{array}{ll}
a_t=e^{-v}\nabla\cdot(e^{v}\nabla a)+f(a,b,v,c), & x\in\Omega,t>0,\\
b_t=e^{- v}\nabla\cdot(e^{v}\nabla b)+g(a,b,v,c), &  x\in\Omega,t>0,\\
v_t=- (ae^{ v}
+ be^{ v})v, &   x\in\Omega,t>0,\\
z_t=D_z \Delta z-z-uz+\beta w, &  x\in\Omega,t>0,\\
\displaystyle\frac{ \partial a}{\partial\nu}=\frac{ \partial b}{\partial\nu}=\frac{ \partial z}{\partial\nu}=0,
  &   x\in\partial\Omega,t>0,\\
a(x,0)=u_0(x) e^{-v_0(x)}, ~b(x,0)=w_0(x) e^{-v_0(x)}, &  x\in \Omega,\\
 v_0(x,0)= v_0(x),~~ z(x,0)=z_0(x), &  x\in \Omega
  \end{array}
 \right.
 \end{equation}
with
$$
f(a,b,v,c):=- az+a(a+b)e^{v}v+\mu a(1-e^v a)
$$
as well as
$$
g(a,b,v,c):=- b+ az+b(a+b)e^{v}v.
$$
Note that \eqref{1.3} and \eqref{2.1} are equivalent in the framework of the classical
solution. In this framework the following statements on local-in-time existence and a convenient extensibility
criteria of the classical solution to \eqref{2.1} can be proved by slightly adapting the arguments detailed in\cite{PW3MAS}.
\begin{lemma}\label{lemma2.1}
Let $\Omega\subset \mathbb{R}^3$ be a bounded domain with smooth boundary; $D_z$ and $\beta$ are positive, and initial data
$(u_0,v_0,w_0,z_0)$ satisfies \eqref{1.4}. Then there exist $T_{max}\in (0,\infty]$ and a  unique   quadruple
$(a,v,b,z)\in C^{2,1}(\overline{\Omega}\times (0,T_{max}))^4 \cap C^0(\overline{\Omega}\times [0,T_{max}))^4$	
which solves \eqref{2.1}  classically in $\Omega\times (0,T_{max})$, and are such that  if $T_{max} <\infty$,
then
\begin{equation}\label{2.2}
\displaystyle\limsup_{t\nearrow T_{max}}\|a(\cdot, t)\|_{L^\infty(\Omega)}+\|b(\cdot, t)\|_{L^\infty(\Omega)}+\|z(\cdot, t)\|_{L^\infty(\Omega)}
+\|\nabla v(\cdot,t)\|_{L^5(\Omega)} \rightarrow \infty.
\end{equation}
\end{lemma}

\section{ A priori estimates for solutions }
In order to set the frame for the self-map type argument under the assumption of $\beta$ in Theorem \ref{T1.1}, it seems necessary for us
 to formulate the following elementary observation.
 \begin{lemma}\label{lemma3.1} Let
 $\beta \max \{1,\|u_0\|_{L^{\infty}(\Omega)}\}< 1+
(1+\frac1{\min_{x\in \Omega}u_0(x)})^{-1}$. Then there exist $K=K(\beta,u_0)>0$, $\delta\in (0,1)$ and $\varepsilon_0\in (0,1)$ such that
for all $0<\varepsilon<\varepsilon_0$, we have
\begin{equation}\label{3.1}
\begin{array}{l}
\displaystyle\frac1{1-\delta}(\max \{\frac{\mu+\varepsilon e}{\mu-\varepsilon e},\|u_0\|_{L^{\infty}(\Omega)}\}+ \sqrt{\varepsilon}e
\max \{3,\|u_0\|_{L^{\infty}(\Omega)}+1\}
)\\[2mm]
<K<\displaystyle\frac{1}{\beta e^\varepsilon}  (1+
(\frac{\mu e^\varepsilon}{\mu-\sqrt{\varepsilon}}+\frac{e^{\frac {\sqrt{\varepsilon}}{\delta}+\varepsilon}}
{\min_{x\in \Omega}u_0(x)})^{-1}-\delta)(1-\frac 32 \sqrt{\varepsilon}).
\end{array}
\end{equation}
  \end{lemma}
    \begin{proof}
By the hypothesis
 $\beta \max \{\|u_0\|_{L^{\infty}(\Omega)},1 \}< 1+
(1+\frac1{\min_{x\in \Omega}u_0(x)})^{-1}$, one can see  that
$$ K=\frac12\left(\max \{1,\|u_0\|_{L^{\infty}(\Omega)}\}+\frac  { 1+
(1+\frac1{\min_{x\in \Omega}u_0(x)})^{-1}  }\beta \right)
$$
satisfies
\begin{equation}\label{3.2}
\max \{1,\|u_0\|_{L^{\infty}(\Omega)}\}<K<\frac  { 1+
(1+\frac1{\min_{x\in \Omega}u_0(x)})^{-1}  }\beta.
\end{equation}
Therefore, the claim follows by means of an argument based on
continuous dependence.
  \end{proof}

 Now  we further take $\varepsilon\in (0,1)$  small enough such that
\begin{equation}\label{3.3}
 \varepsilon<\varepsilon_1:=\min\{ \frac{\beta^2} 4, \frac{\mu}{3e}\}
\end{equation}
  and  henceforth suppose that  the initial data $(u_0,v_0,w_0,z_0)$
fulfills \eqref{1.4}--\eqref{z1.7}, and then consider  the set
\begin{equation}\label{3.4}
\mathcal{S}=\left\{T_1>0\left|\right.~\|z(\cdot,t)\|_{L^{\infty}(\Omega)}<\sqrt{\varepsilon} e^{-\delta t}~~~\hbox{for all $t\in (0,T_1\cap T_{max})$}\right\}
\end{equation}
which is not empty due to \eqref{z1.7}. In particular,  $T=\sup \mathcal{S}\in(0, T_{max}]$,  the maximal  interval on  which
$\|z(\cdot,t)\|_{L^{\infty}(\Omega)}< \sqrt{\varepsilon} e^{-\delta t}$ holds,  is well-defined.

As the first step of the proof of Theorem \ref{T1.1}, we make sure that if $T_{max}<\infty$, then
\begin{equation}\label{3.5}
T=T_{max}.
\end{equation}
To this end, we suppose that $T<T_{max}$ and then derive the pointwise lower-bounded estimate of $a$ on interval $(0,T)$ from the hypothesis included in \eqref{3.4} by a simple comparison argument.
\begin{lemma}\label{lemma3.2}
Let the assumptions in Theorem  \ref{T1.1} hold and  the initial data $(u_0,v_0,w_0,z_0)$
fulfills \eqref{1.4}-- \eqref{z1.7} with  $\varepsilon<\min\{\varepsilon_0,\varepsilon_1\}$.  Then we have
\begin{equation}\label{3.6}
a(x,t)\geq \left\{
\frac{
e^{\frac{\sqrt{\varepsilon}}{\delta}+\epsilon}}{\min_{x\in \Omega}u_0(x)}+\frac{\mu e^{\varepsilon}}{\mu-\sqrt{\varepsilon}}
 \right\}^{-1} \quad \quad \hbox{for all} ~~x\in\Omega, t\in(0,T)	
\end{equation}
as well as
\begin{equation}\label{3.7}
v(x,t)\leq  \varepsilon
\exp\{-(
\frac{
e^{\frac{\sqrt{\varepsilon}}{\delta}+\epsilon}}{\min_{x\in \Omega}u_0(x)}+\frac{\mu e^{\varepsilon}}{\mu-\sqrt{\varepsilon}})^{-1} t\}  \quad \hbox{for all} ~~x\in\Omega, t\in(0,T).	
\end{equation}
\end{lemma}
\begin{proof}
From the definition of $T$ and $v(x,t)\leq v_0(x)\leq \varepsilon$, one can see that for all $(x,t)\in \Omega\times(0,T)$,
\begin{equation}\label{3.8}
\begin{array}{ll}
a_t&\geq e^{-v}\nabla \cdot (e^{v}\nabla a)+(\mu-\sqrt{\varepsilon} e^{-\delta t})a-\mu e^{\varepsilon}a^2 \\[2mm]
 &= \triangle v - \nabla v \cdot \nabla a+(\mu-\sqrt{\varepsilon} e^{-\delta t})a-\mu e^{\varepsilon}a^2.
\end{array}
\end{equation}
Thus, if we let $\underline{a}(t)\in C^1((0,T)) $ denote
the solution of
\begin{equation}\label{3.9}
\left\{
\begin{array}{lll}
\underline{a}_t=(\mu-\sqrt{\varepsilon} e^{-\delta t})\underline{a}-\mu e^{\varepsilon}\underline{a}^2,\quad & t\in (0,T),\\
\underline{a}(0)=a_{0*}:=\min_{x\in \Omega} a_0(x), 	
\end{array}
\right.
\end{equation}
then thanks to explicit solution of the Bernoulli-type initial-value problem (IVP) \eqref{3.9},  the comparison principle asserts   that
\begin{equation}\label{3.10a}
a(x,t)\geq \underline{a}(t)\geq \left\{
\frac{
e^{\frac{\sqrt{\varepsilon}}{\delta}+\epsilon}}{\min_{x\in \Omega}u_0(x)}+\frac{\mu e^{\varepsilon}}{\mu-\sqrt{\varepsilon}}
 \right\}^{-1}
\hbox{for all}~~ (x,t)\in\overline{\Omega}\times[0,T).	
\end{equation}
On the other hand, by the nonnegativity  of functions $b$ and $v$, we have
$
v_t\leq -av
$
which along with \eqref{3.10a} yields \eqref{3.7} immediately.
\end{proof}

Based on above estimates,  one can obtain suitable upper bounds
for $b$ as well as for $a$ by  adequately exploiting
a cooperative parabolic system.

\begin{lemma}\label{lemma3.3} Suppose that  the assumptions in Lemma  \ref{lemma3.2} hold, then
$$
a(x,t)\leq \max\{\|u_0\|_{L^\infty(\Omega)},2\},~~b(x,t)\leq K\sqrt{\varepsilon} e^{-\delta t}
$$
where $K>0,\varepsilon>0$ and $\delta>0 $ are  given in Lemma \ref{lemma3.1}.
\end{lemma}

\begin{proof} Let $a_*= \{
\frac{
e^{\frac{\sqrt{\varepsilon}}{\delta}+\epsilon}}{\min_{x\in \Omega}u_0(x)}+\frac{\mu e^{\varepsilon}}{\mu-\sqrt{\varepsilon}}
\}^{-1}$.  Then due to the positivity of $a$, $b$ and $z$, it follows from \eqref{3.4}, \eqref{3.7} that
 \begin{equation}\label{3.10}
\left\{
\begin{array}{lll}
a_t\leq e^{-v}\nabla \cdot(e^{v}\nabla a)+\varepsilon e^{\varepsilon}a(a+b) e^{-a_*t}+\mu a-\mu a^2,\quad & x\in\Omega,t\in (0,T),\\[2mm]
b_t\leq e^{-v}\nabla \cdot(e^{v}\nabla b)-b+\varepsilon e^{\varepsilon}b(a+b) e^{-a_*t}+
\sqrt{\varepsilon} a e^{-\delta t}, &x\in\Omega,t\in (0,T),
\end{array}
\right.	
\end{equation}
 and thereby  $(a,b)$ is a sub-solution of the cooperative  parabolic system
 \begin{equation}\label{3.11}
\left\{
\begin{array}{lll}
\tilde{a}_t= e^{-v}\nabla \cdot(e^{v}\nabla \tilde{a})+\varepsilon e^{\varepsilon}\tilde{a}(\tilde{a}+\tilde{b}) e^{-a_*t}+\mu\tilde{a}-\mu \tilde{a}^2,\quad & x\in\Omega,t\in (0,T),\\[2mm]
\tilde{b}_t= e^{-v}\nabla \cdot(e^{v}\nabla \tilde{b})-\tilde{b}+\varepsilon e^{\varepsilon}\tilde{b}(\tilde{a}+\tilde{b}) e^{-a_*t}+
\sqrt{\varepsilon} \tilde{a}e^{-\delta t}, &x\in\Omega,t\in (0,T),\\[2mm]
\tilde{a}(x,0)=u_0(x)e^{-v_0(x)},~~\tilde{b}(x,0)=w_0(x)e^{-v_0(x)}.
\end{array}
\right.	
\end{equation}
 In order to construct a appropriate super-solution $(\hat{a},\hat{b})$ to problem \eqref{3.11}, we let
 \begin{equation}\label{3.12}
 \hat{a}(x,t)=\varphi(t), ~~\hat{b}(x,t)=K\varepsilon^{\frac 12} e^{-\delta t} ~~~\hbox{for}~~ x\in \Omega, t\in (0,T),
\end{equation}
 where  $K$ is given in Lemma \ref{lemma3.1}, and $\varphi\in C^1([0,T))$ is the solution of  the Bernoulli-type IBV
 \begin{equation}\label{3.13}
 \left\{
 \begin{array}{l}
 \varphi'(t)=(\varepsilon e+\mu)\varphi(t)-(\mu-\varepsilon e)\varphi^2(t),\\
 \varphi(0)=\|u_0\|_{L^\infty(\Omega)}
\end{array}
\right.
\end{equation}
which satisfies
 \begin{equation}\label{3.14}
 \varphi(t)\leq  \max\{\|u_0\|_{L^\infty(\Omega)}, \frac{\varepsilon e+\mu}{\mu-\varepsilon e}\}
\end{equation}
according to the explicit solution thereof.

Thanks to the latter,  $(\widehat{a},\widehat{b} )$ is  actually a super-solution of \eqref{3.12}. Indeed, due to \eqref{3.3},
$\frac{\varepsilon e+\mu}{\mu-\varepsilon e}<2$ and  $K\sqrt{\varepsilon}<1$. Hence from \eqref{3.14} and \eqref{3.1}, it follows that
 \begin{equation}\label{3.15}
 \begin{array}{ll}
 &K(\varphi(t)+K\varepsilon^{\frac 12} e^{-\delta t})e^{\varepsilon}\varepsilon+\varphi(t)\\[2mm]
 \leq  &  K(\max\{\|u_0\|_{L^\infty(\Omega)}, \frac{\varepsilon e+\mu}{\mu-\varepsilon e}\}
 +K\varepsilon^{\frac 12} )e^{\varepsilon}\varepsilon+  \max\{\|u_0\|_{L^\infty(\Omega)}, \frac{\varepsilon e+\mu}{\mu-\varepsilon e}\}  \\[2mm]
 \leq  &  K\varepsilon e\max\{\|u_0\|_{L^\infty(\Omega)}+1, 3\}  + \max\{\|u_0\|_{L^\infty(\Omega)}, \frac{\varepsilon e+\mu}{\mu-\varepsilon e}\}
 \\[2mm]
  \leq  &   \sqrt{\varepsilon} e\max\{\|u_0\|_{L^\infty(\Omega)}+1, 3\}  + \max\{\|u_0\|_{L^\infty(\Omega)}, \frac{\varepsilon e+\mu}{\mu-\varepsilon e}\} \\[2mm]
  \leq & K(1-\delta),
 \end{array}
  \end{equation}
  which implies that
 $$
  \hat{b}_t\geq e^{-v}\nabla \cdot(e^v\nabla\hat{b})-\hat{b}+\varepsilon e^{\varepsilon}\hat{b}(\hat{a}+\hat{b}) e^{-a_*t}+
\sqrt{\varepsilon} \hat{a}e^{-\delta t},x\in\Omega,t\in (0,T).
 $$
In addition, due to \eqref{w1.6} and \eqref{3.2},
$$\hat{b}(x,0)\geq \|w_0\|_{L^\infty(\Omega)}\geq  w_0(x)e^{-v_0(x)}.
$$
Apart from that, in light of \eqref{3.13}, we have
$$
\hat{a}_t\geq e^{-v}\nabla \cdot(e^{v}\nabla \hat{a})+\varepsilon e^{\varepsilon}\hat{a}(\hat{a}+\hat{b}) e^{-a_*t}+\mu\hat{a}-\mu \hat{a}^2,
~~~\hat{a}(x,0)\geq u_0(x)e^{-v_0(x)}.$$
Therefore $ (\hat{a}, \hat{b})$ is readily verified to be a super-solution of problem \eqref{3.11}, and thereby  by
comparison principle  $$ a(x,t)\leq \hat{a}(x,t)\leq \max\{\|u_0\|_{L^\infty(\Omega)},2\},~~~
 b(x,t)\leq \hat{b}(x,t)= K\sqrt{\varepsilon} e^{-\delta t}. $$
\end{proof}

 Based on the outcomes in Lemma  \ref{lemma3.2} and  Lemma \ref{lemma3.3}, we can verify that actually $T=T_{max}$ provided
 $T_{max}<\infty$, which is stated as follows.

 \begin{lemma}\label{lemma3.4}Let the assumptions in Theorem  \ref{T1.1} hold and  the initial data $(u_0,v_0,w_0,z_0)$
fulfills \eqref{1.4}--\eqref{z1.7} with  $\varepsilon<\min\{\varepsilon_0,\varepsilon_1\}$.  Then
 \begin{equation}\label{3.16}
 z(x,t)\leq (\sqrt{\varepsilon}-\frac \varepsilon 2)e^{-\delta t} \quad \hbox{for all}~~ x\in\Omega, t\in(0,T),
 \end{equation}
as well as $T=T_{max}$ provided
 $T_{max}<\infty$.
 \end{lemma}
\begin{proof}
According to  Lemma  \ref{lemma3.2} and  Lemma \ref{lemma3.3}, we have
 $$u(x,t)\geq a(x,t)\geq
 a_*:= \{
\frac{
e^{\frac{\sqrt{\varepsilon}}{\delta}+\epsilon}}{\min_{x\in \Omega}u_0(x)}+\frac{\mu e^{\varepsilon}}{\mu-\sqrt{\varepsilon}}
\}^{-1}$$
 and
 $$ w(x,t)\leq b(x,t) e^{\varepsilon}\leq  K\sqrt{\varepsilon}  e^{\varepsilon} e^{-\delta t}. $$
 Hence in light of the fourth equation of \eqref{2.1}, $z$ satisfies
 \begin{equation}\label{3.17}
 z_t\leq D_z \Delta z- a_* z- z+\beta  K\sqrt{\varepsilon}  e^{\varepsilon} e^{-\delta t}, ~~ x\in\Omega, 0<t<T.
 \end{equation}

 Now we  may compare $z$ to
spatially homogeneous functions having a supersolution property with regard to the parabolic
operator in \eqref{3.17}. Indeed,
 if $\hat{z}(t)$ denotes the solution of the initial-value problem
 \begin{equation*}\left\{
   \begin{array}{l}
y_t=- (a_*+1) y+\beta  K\sqrt{\varepsilon}  e^{\varepsilon} e^{-\delta t},~~ 0<t<T,\\[2mm]
y(0)=\|z_0\|_{L^{\infty}(\Omega)},
\end{array}
\right.
 \end{equation*}
   then from the comparison principle we infer that
 \begin{equation}\label{3.18}
  \begin{array}{ll}
z(x,t)\leq \hat{z}(t)&= \|z_0\|_{L^{\infty}(\Omega)}e^{-(a_*+1)t}+ \beta  K\sqrt{\varepsilon}  e^{\varepsilon} \displaystyle\int^t_0
e^{-(a_*+1)(t-s)}  e^{-\delta s}ds\\[2mm]
&\leq \varepsilon e^{-(a_*+1)t}+\displaystyle\frac{\beta  K\sqrt{\varepsilon}  e^{\varepsilon}}{1+a_*-\delta} e^{-\delta t}\\
&\leq (\sqrt{\varepsilon}- \displaystyle\frac{ \varepsilon}2) e^{-\delta t},
\end{array}
 \end{equation}
 where  inequality \eqref{3.1} is used  in the last inequality.

 At this position, supposed that  $T<T_{max}$, then according to the definition of $T$ and continuity of $z$, we have
$\|z(\cdot,T)\|_{L^{\infty}(\Omega)}=\sqrt{\varepsilon} e^{-\delta T}$, which contradicts with   \eqref{3.18} and thereby shows that   actually $T=T_{max}$.

\end{proof}

\section{Global solvability}
The purpose of this section is to show that under the assumptions in Theorem \ref{T1.1}, $(a,v,b,z)$ is actually  a global solution of \eqref{2.1}.
 To this end, in light of \eqref{2.2} and the  outcomes of Lemma \ref{lemma3.3} and Lemma \ref{lemma3.4},  we need to establish a bound for $\nabla v$ with respect to the norm in $L^5(\Omega)$, rather $L^4(\Omega)$ in the two-dimensional setting. With respect to this, we first derive the spatial-temporal estimates of $\Delta a$  through the testing processes, thanks to the $L^\infty(\Omega)$
    -estimates of $a,b$ and $z$ just asserted.
\begin{lemma}\label{lemma4.1} Let $(a,v,b,z)$ be the classical solution of \eqref{2.1} on $(0,T_{max})$ obtained in previous section. Then if $T_{max}<\infty$, one can find $C>0$ fulfilling
\begin{equation}\label{4.1}
\int^t_0\int_{\Omega}(|\triangle a|^2+|\triangle b|^2)\leq C ~~~~~ \hbox{for all}~~ t\in(0,T_{max})
\end{equation}
as well as
\begin{equation}\label{4.2}
\int_{\Omega}|\nabla v(\cdot,t)|^4\leq C ~~~~~ \hbox{for all}~~ t\in(0,T_{max}).	
\end{equation}
\end{lemma}

\begin{proof}
According to the outcomes of Lemma \ref{lemma3.3} and Lemma \ref{lemma3.4}, we have
\begin{equation}\label{4.3}
a(x,t)\leq c_1:=\max\{\|u_0\|_{L^\infty(\Omega)},2\},~~ v(x,t)\leq 1,~~ b(x,t)\leq 1 ~~ \hbox{and}~ z(x,t)\leq 1,	
\end{equation}
which definitely entail that
$$
|f(a,b,v,z)|\leq c_2:=c_1+c_1(c_1+1)e +\mu c_1^2 e.
$$	
Multiplying  the equation
$$
a_t=\Delta a+\nabla v\cdot\nabla a+f(a,b,v,z)
$$
by $-\Delta a$ and integrating by parts, we obtain that
\begin{equation*}\begin{array}{rl}
&\displaystyle\frac{1}{2}\frac{d}{dt}\int_{\Omega}|\nabla a|^2+\int_{\Omega}|\Delta a|^2\\
=&\displaystyle -\int_{\Omega}\Delta a\nabla v\cdot \nabla a-\int_{\Omega}f\Delta a\\
\leq &\displaystyle \frac{1}{2}\int_{\Omega}|\Delta a|^2+\int_{\Omega}|\nabla v\cdot\nabla a|^2+c_2^2|\Omega|.
\end{array}
\end{equation*}
Here combining the Gagliardo--Nirenberg inequality with the standard elliptic regularity, one can find  $c_3>0$ such that
$$
\|\nabla \varphi\|^2_{L^{4}(\Omega)}\leq c_3\|\Delta \varphi\|_{L^{2}(\Omega)}\|\varphi \|_{L^{\infty}(\Omega)}\quad \hbox{for all $\varphi\in C^{2}(\overline\Omega)$ with $\frac{\partial \varphi}{\partial\nu}=0$}
$$
which together with \eqref{4.3} shows that
\begin{equation}\label{4.4}
\int_{\Omega}|\Delta a |^2	\geq \frac{1}{c^2_1c_3^2}\int_{\Omega}|\nabla a|^4 \quad \hbox{for all $0<t<T_{max}$}.
\end{equation}
Therefore applying  \eqref{4.3} and the Young inequality, we have
\begin{equation}\label{4.5}
\frac{d}{dt}\int_{\Omega}|\nabla a|^2+c_4\int_{\Omega}|\nabla a|^4+c_4\int_{\Omega}|\Delta a|^2\leq c_5\int_{\Omega}|\nabla v|^4+c_5	
\end{equation}
for some constants  $c_4>0, c_5>0$.
The similar argument allows us to find  $c_6>0$ and $c_7>0$ fulfilling
\begin{equation}\label{4.6}
\frac{d}{dt}\int_{\Omega}|\nabla b|^2+c_6\int_{\Omega}|\nabla b|^4+c_6\int_{\Omega}|\triangle b|^2\leq c_7	\int_{\Omega}|\nabla v|^4+c_7.
\end{equation}
To compensate the right summand of \eqref{4.5} and \eqref{4.6}, we test the $v$-equation in \eqref{2.1}   against
$|\nabla v|^2\nabla v$ to obtain that for all $0<t<T_{max}$,
\begin{equation}\label{4.7}
\begin{array}{rl}
\displaystyle\frac 14\frac{d}{dt}\int_{\Omega} |\nabla v|^4=
&-\displaystyle\int_{\Omega} v e^v |\nabla v|^2\nabla v\cdot (\nabla a+\nabla b)-\int_{\Omega}e^v (v+1)(a+b)|\nabla v|^4\\[2mm]
\leq & -\displaystyle\int_{\Omega}ve^v |\nabla v|^2\nabla v\cdot (\nabla a+\nabla b)-\displaystyle\int_{\Omega}ae^{v}|\nabla v|^4
\end{array}
\end{equation}
Recalling the  lower bound for $a$ in \eqref{3.6}, one  can find  $c_8>0$ such that
$$
\int_{\Omega}ae^{v}|\nabla v|^4\geq c_{8}\int_{\Omega}|\nabla v|^4.
$$
So we combine \eqref{4.7} and \eqref{4.3} with the Young inequality to see that
\begin{equation}\label{4.8}
\frac{d}{dt}\int_{\Omega}|\nabla v|^4+c_{9}\int_{\Omega}|\nabla v|^4\leq
 c_{10}(\int_{\Omega}|\nabla a|^4+\int_{\Omega}|\nabla b|^4)
\end{equation}
for some $c_9>0, c_{10}>0$.

Therefore,  through an appropriate combination of \eqref{4.5}, \eqref{4.6} and \eqref{4.8}, one can find
  $\eta>0$, $c_{11}>0$ and $c_{12}>0$  such that
\begin{equation}\label{4.9}
\begin{array}{rl}
&\displaystyle\frac{d}{dt}\left(\int_{\Omega}|\nabla a|^2+\int_{\Omega}|\nabla b|^2+\eta\int_{\Omega}|\nabla v|^4\right)
+c_{11} \int_{\Omega}(|\triangle a|^2+|\triangle b|^2)\\[3mm]
\leq &c_{12}\displaystyle\left(\int_{\Omega}|\nabla a|^2+\int_{\Omega}|\nabla b|^2+\eta\int_{\Omega}|\nabla v|^4\right)+c_{12}.
\end{array}
\end{equation}
At this position, writing $y(t):=\int_{\Omega}|\nabla a(\cdot,t)|^2+\int_{\Omega}|\nabla b(\cdot,t)|^2+\eta\int_{\Omega}|\nabla v(\cdot,t)|^4+1$, we
 then have
\begin{equation*}
y^{\prime}(t)+c_{11} \int_{\Omega}(|\triangle a|^2+|\triangle b|^2) \leq  c_{12} y(t) 	
\end{equation*}
for all $0<t<T_{max}$, which  leads to
$$
y(t)\leq y_0 e^{c_{12}T_{\max}}
$$
as well as
\begin{equation*}
\int^t_0\int_{\Omega}(|\triangle a|^2+|\triangle b|^2) \leq  C(T_{\max})
\end{equation*}
for all $0<t<T_{max}$, and thereby completes the proof.
\end{proof}
Now making efficient use of the spatial-temporal estimates of $\Delta a$ provided by Lemma \ref{lemma4.1}, one can derive the further regularity property of $\nabla v$ in the three-dimensional setting beyond that in \eqref{4.2}.

\begin{lemma}\label{lemma4.2}
Assume that $T_{max}<\infty$, one can find $C>0$ fulfilling
\begin{equation}\label{4.10}
\int_{\Omega}|\nabla v(\cdot,t)|^5\leq C ~~~~~ \hbox{for all}~~ t\in(0,T_{max}).	
\end{equation}
\end{lemma}
\begin{proof} Multiplying the third equation in \eqref{2.1} by
$|\nabla v|^3\nabla v$, integrating by parts and applying Young's inequality, we obtain that for all $0<t<T_{max}$,
\begin{equation}\label{4.11}
\begin{array}{rl}
\displaystyle\frac 15\frac{d}{dt}\int_{\Omega} |\nabla v|^5=
&-\displaystyle\int_{\Omega} v e^v |\nabla v|^3\nabla v\cdot (\nabla a+\nabla b)-\int_{\Omega}e^v (v+1)(a+b)|\nabla v|^4\\[2mm]
\leq & -\displaystyle\int_{\Omega}ve^v |\nabla v|^3\nabla v\cdot (\nabla a+\nabla b)\\[2mm]
\leq &  e\left\{\displaystyle\int_{\Omega} |\nabla v|^5\right\}^{\frac 45 }
(\displaystyle \left\{\int_{\Omega}|\nabla a|^5\right\} ^{\frac 15 }+\displaystyle\left\{\int_{\Omega}|\nabla b|^5\right\}^{\frac 15 }).
\end{array}
\end{equation}

Now we see that   $y(t):=\int_{\Omega} |\nabla v(\cdot,t)|^5$, $t\in(0,T_{max})$  satisfies
$$
y'(t)\leq 5e  (\|\nabla a(\cdot,t)\|_{L^5(\Omega)}+\|\nabla b(\cdot,t)
\|_{L^5(\Omega)})
 y^{\frac 45 }(t)
$$
and hence
\begin{equation}\label{4.12}
\|\nabla v(\cdot,t)\|_{L^5(\Omega)}\leq \|\nabla v_0\|_{L^5(\Omega)}+e\int^{t}_0
(\|\nabla a(\cdot,s)\|_{L^5(\Omega)}+\|\nabla b(\cdot,s)
\|_{L^5(\Omega)})ds.
\end{equation}
Furthermore, as
$$
\|\nabla \varphi\|_{L^{5}(\Omega)}\leq c_1\|\Delta \varphi\|^{\frac 45 }_{L^{2}(\Omega)}\|\varphi \|^{\frac 15 }_{L^{\infty}(\Omega)}\quad \hbox{for all $\varphi\in C^{2}(\overline\Omega)$ with $\frac{\partial \varphi}{\partial\nu}=0$}
$$
with constant  $c_1>0$,  we get
\begin{equation*}
\begin{array}{rl}
&\|\nabla v(\cdot,t)\|_{L^5(\Omega)}\\
\leq &\|\nabla v_0\|_{L^5(\Omega)}+c_2\displaystyle\int^{t}_0
(\|\Delta a(\cdot,s)\|^2_{L^2(\Omega)}+\|\Delta b(\cdot,s)
\|^2_{L^2(\Omega)}+1)ds
\end{array}
\end{equation*}
for  $c_2>0$, which along with \eqref{4.1}  yields \eqref{4.10} immediately.
\end{proof}

\begin{corollary}\label{corollary4.1} Under the assumptions in Theorem \ref{T1.1}, problem \eqref{1.3}
admits a unique global solution $(u,v,w,z)$.
\end{corollary}
\begin{proof}
According to the equivalence of \eqref{1.3}  and \eqref{2.1} in the considered framework
of classical solutions, the global existence readily results  from  Lemma \ref{lemma3.2}, Lemma \ref{lemma3.3}, Lemma \ref{lemma4.2} and
the  extensibility criteria \eqref{2.2} in Lemma 2.1.
\end{proof}

\section{Exponential decay of $u-1$}
Throughout this section, $(a,v,b,z)$ and $(u,v,w,z)$  are the global classical solution of \eqref{2.1} and  \eqref{1.3}, respectively.
Combining the outcomes of Lemma \ref{lemma3.2}, Lemma \ref{lemma3.3} and Lemma \ref{lemma3.4} with  Corollary \ref{corollary4.1}, we can see that
\eqref{1.9}--\eqref{1.11} in Theorem \ref{T1.1} have been proved already, and thereby only need to show that \eqref{1.8} is valid.

The decay property \eqref{1.8} will be shown by means of the bootstrap method adapted from\cite{WX-preprint},  the latter is  concerned with the two-dimensional version of \eqref{1.2} with $\beta<1$.  As the start point  of the derivation of \eqref{1.8},  we establish the
following basic stabilization feature of $a(=e^{- v}u)$
 upon the  decay property of $v$ with respect to $L^\infty(\Omega)$.
\begin{lemma}\label{lemma5.1}
Let the assumptions in Theorem \ref{T1.1} hold. Then
\begin{equation}\label{5.1}
\int^{\infty}_0\int_{\Omega}\frac{|\nabla a|^2}{a^2}< \infty
	\end{equation}
and
\begin{equation}\label{5.2}
\int^{\infty}_0\int_{\Omega}(u-1)^2< \infty.
	\end{equation}
\end{lemma}
\begin{proof}
In view of $s-1-\log s>0$ for all $s>0$ and $v_t<0$, we can conclude that
\begin{equation}\label{5.3}
\begin{array}{rl}
&\displaystyle\frac{d}{dt}\int_{\Omega}e^v(a-1-\log a)\\
=& \displaystyle\int_{\Omega}e^{v}(a-1-\log a)v_t+\int_{\Omega}\frac{a-1}{a} e^v a_t\\[2mm]
\leq& \displaystyle-\int_{\Omega}e^v\frac{|\nabla a|^2}{a^2}+\mu\int_{\Omega}	e^v (a-1)(1-u)\\
& +\displaystyle\int_{\Omega}e^v (a-1)(u+w)v-
\displaystyle\int_{\Omega}e^v z(a-1).
\end{array}
\end{equation}
Here by Young's inequality,
\begin{equation}\label{5.4}
\begin{array}{rl}
(1-a)(1-u)&=(1-u)^2+(u-a)(1-u)\\
&\geq \displaystyle\frac 12 (1-u)^2- a^2(e^{v}-1)^2.
\end{array}
\end{equation}
Since $e^s\leq 1+2s$ for all $s\in [0,log2]$,  \eqref{3.7} allows us to fix a $t_1>1$ suitable large such that for all $t\geq t_1$,
$
(e^v-1)^2\leq 4 v^2,
$
which along with \eqref{5.4} implies that
$$
(1-a)(1-u)\geq \displaystyle\frac 12 (1-u)^2-4 a^2 v^2
$$
for $t\geq t_1$. Hence thanks to the outcomes of Lemma \ref{lemma3.3} and Lemma \ref{lemma3.4}, one can
obtain  from \eqref{5.3} that for all $t\geq t_1$
\begin{equation*}
\begin{array}{rl}
&\displaystyle \frac{d}{dt}\int_{\Omega}e^{v}(a-1-\log a)+\int_{\Omega}\frac{|\nabla a|^2}{a^2}+
\frac{\mu}2 \int_{\Omega}(u-1)	 ^2\\[2mm]
\leq &\displaystyle
c_1\int_{\Omega}z+c_1\int_{\Omega}v
\end{array}
\end{equation*}
with  $c_1>0$.
Upon a time integration over $(t_1,t)$, the latter  arrives  at
\begin{equation*}
\begin{array}{rl}
&\displaystyle \int^t_{t_1}\int_{\Omega}\frac{|\nabla a|^2}{a^2}+\frac \mu 2 \int^{t}_{t_1} \int_{\Omega}(u-1)^2\\[2mm]
\leq & c_1 \displaystyle\int^{t}_{t_1}\int_{\Omega}z+c_1\int^t_{t_1}\int_{\Omega}v+
\int_{\Omega}e^{v(\cdot,t_1)}(a(\cdot,t_1)-1-\log a (\cdot,t_1)),
\end{array}
\end{equation*}
which together with  \eqref{3.6} and \eqref{3.7} makes sure that  both \eqref{5.1} and \eqref{5.2} hold.	
\end{proof}

In order to make sure that \eqref{5.2} implies the time decay property  of $u-1$, it seems desirable to consider the exponential decay
 properties of  $\int_{\Omega}|\nabla v(\cdot,t)|^2$, rather than the integrability  of   $a_{t}$ in  $L^2((0,\infty); L^2(\Omega))$.  As the first step toward this, we first show the convergence  of integral $\int^{\infty}_0\int_{\Omega}|\nabla v|^2$, which is stated below.
\begin{lemma}\label{lemma5.2}
Suppose that  the assumptions in Theorem  \ref{T1.1} hold. Then we have
\begin{equation}\label{5.5}
\int^{\infty}_0\int_{\Omega}|\nabla v|^2<\infty.	
\end{equation}	
\end{lemma}

\begin{proof}
Testing the second equation in \eqref{2.1} against $e^{ v}b$ and integrating by parts, we get
\begin{eqnarray*}
\frac{d}{dt}\int_{\Omega}e^{v}b^2+2\int_{\Omega}e^{ v}|\nabla b|^2+2 \int_{\Omega}e^{ v}b^2\\
=2\int_{\Omega}abz e^{ v}+2\int_{\Omega}e^{v}b^2( a e^{ v}+ b e^v)v,	
\end{eqnarray*}
which together with the global-in-time boundedness property of $a$ and $z$, implies that	
\begin{eqnarray*}
\frac{d}{dt}\int_{\Omega}e^{ v}b^2+\int_{\Omega} |\nabla b|^2+2\int_{\Omega}e^{v}b^2
\leq c_1\int_{\Omega}b	
\end{eqnarray*}
for some $c_1>0$.
Hence according to Lemma \ref{lemma3.3}, we can get
\begin{equation}\label{5.6}
\int^{\infty}_0\int_{\Omega}|\nabla b|^2<\infty.	
\end{equation}
Now since
 $$
 \nabla v_t=-( \nabla u + \nabla w )v-( u +   w ) \nabla v,
 $$
a direct computation shows that
\begin{equation*}
\begin{array}{rl}
&\displaystyle \frac{1}{2}\frac{d}{dt}\int_{\Omega}|\nabla v|^2+\int_{\Omega}(u+w)|\nabla v|^2\\
=&-\displaystyle \int_{\Omega}v e^{ v} b|\nabla v|^2-
\displaystyle\int_{\Omega}v e^{ v}\nabla v\cdot \nabla b-
\displaystyle\int_{\Omega}v\nabla v\cdot \nabla u\\
\leq&-\displaystyle\int_{\Omega}v e^{v}\nabla v\cdot \nabla b
- \displaystyle\int_{\Omega}v   e^{ v}\nabla v\cdot \nabla a.	
\end{array}
\end{equation*}
Therefore, thanks to  the point-wise lower bound in \eqref{3.6}, and the boundedness of $a$ and $v$,  we can find
find  $c_2>0$ such that
\begin{equation}\label{5.7}
\frac{d}{dt}\int_{\Omega}|\nabla v|^2+a_* \int_{\Omega}|\nabla v|^2\leq c_2
(
\int_{\Omega}|\nabla b|^2+
\int_{\Omega}\frac{|\nabla a|^2}{a^2}
).
\end{equation}
Let $y(t):=\int_{\Omega}|\nabla v|^2$ and  $h(t):= c_2
(
\int_{\Omega}|\nabla b|^2+
\int_{\Omega}\frac{|\nabla a|^2}{a^2}
)$. Then we infer from \eqref{5.7} that for $t>0$,
\begin{equation}\label{5.8}
y'(t)+a_* y(t)\leq h(t),
\end{equation}
where \eqref{5.6} and  \eqref{5.1} warrant the existence of $c_3>0$ such that
\begin{equation}\label{5.9}
\int^{\infty}_0 h(s)ds\leq c_3.
\end{equation}
Therefore thanks to \eqref{5.9}, an integration of \eqref{5.8} yields \eqref{5.5}.
\end{proof}

Upon estimates \eqref{5.5}, 
\eqref{5.6} and  \eqref{5.1}, we make use of the explicit expression  of
$\nabla v$  to verify the exponential decay property of $\int_{\Omega}|\nabla v|^2$.

\begin{lemma}\label{lemma5.3}
There exists $C>0$
such that
\begin{equation}\label{5.10}
\int_{\Omega}|\nabla v|^2\leq C(t+1)e^{-2a_* t}~~\hbox{for all $t>0$,}	
\end{equation}
where $a_*= \{
\frac{
e^{\frac{\sqrt{\varepsilon}}{\delta}+\epsilon}}{\min_{x\in \Omega}u_0(x)}+\frac{\mu e^{\varepsilon}}{\mu-\sqrt{\varepsilon}}
\}^{-1}$.	
\end{lemma}
\begin{proof}
According to the second equation in \eqref{1.3}, we have
\begin{equation*}
\nabla v(\cdot,t)= \nabla v(\cdot,0)e^{-\int^{t}_0(u+w)(\cdot,s)ds} -v(\cdot,0)e^{-\int^{t}_0(u+w)(\cdot,s)ds}
\int^{t}_0(\nabla u(\cdot,s)+\nabla w(\cdot,s))ds,	
\end{equation*}
which, together  with the fact that $w\geq 0$, $u=ae^v\geq a_*$ due to \eqref{3.6} and thereby $u(x,t)+w(x,t)\geq a_* $ for $x\in \Omega, t>0$,
implies that
\begin{equation}\label{5.11}
\begin{array}{ll}
&\displaystyle\int_{\Omega}|\nabla v(\cdot,t)|^2\\[3mm]
\leq&
2e^{-2a_* t}\|\nabla v_0\|^2_{L^2(\Omega)}+4 te^{-2a_* t}\displaystyle
(\int^{t}_0\int_{\Omega}|\nabla u|^2ds+\int^{t}_0\int_{\Omega}|\nabla w|^2ds).
\end{array}
\end{equation}	
Further, since
\begin{equation*}
|\nabla w|=|e^v \nabla b+ e^vb \nabla v|\leq e(|\nabla v|b+|\nabla b|)	
\end{equation*}
as well as \begin{equation*}
|\nabla u| =|e^v \nabla a+ e^v a \nabla a|\leq  e(|\nabla v|a+|\nabla a|),	
\end{equation*}
we  infer from \eqref{5.11} that there exists $c_1>0$ such that for $t>0$,
\begin{equation}\label{4.21}
\begin{array}{rl}
\displaystyle\int_{\Omega}|\nabla v(\cdot,t)|^2
&\leq
c_1e^{-2a_*t}+c_1 te^{-2a_* t}
\displaystyle\int^{t}_0\int_{\Omega}(|\nabla b|^2+|\nabla a|^2 +|\nabla v|^2)ds\\[3mm]
&\leq
c_1e^{-2a_* t}+c_1 te^{-2a_*t}
\displaystyle\int^{\infty}_0\int_{\Omega}(|\nabla b|^2+|\nabla a|^2 +|\nabla v|^2)ds
\end{array}
\end{equation}	
and thereby thanks to estimates \eqref{5.5}, \eqref{5.6} and  \eqref{5.1},  we get
$$
\displaystyle\int_{\Omega}|\nabla v(\cdot,t)|^2\leq c_2(t+1)e^{-2a_* t}
$$
with some $c_2>0$.
\end{proof}
On the basis of smoothing estimates for the Neumann heat semigroup on $\Omega$, we merely turn the
 decay information provided by Lemma \ref{lemma5.3} into
the exponential decay of $u-1$  with respect to  the norm in $L^p(\Omega)$ for arbitrary $p<6$.

\begin{lemma}\label{lemma5.4} There exists $\gamma>0$ such that for every $p<6$,
\begin{equation}\label{5.13}
\|u(\cdot,t)-1\|_{L^p(\Omega)}\leq C(p) e^{-\gamma t}
\end{equation}	
with some $C(p)>0$ for all $t>0$.
\end{lemma}
\begin{proof}
Testing the first equation in \eqref{1.3} by $u-1$ and  integrating by parts, we have
\begin{equation*}
\begin{array}{rl}
&\displaystyle\frac{d}{dt}\int_{\Omega}(u-1)^2+2\int_{\Omega}|\nabla u|^2+2\mu\int_{\Omega}u(u-1)^2\\[2mm]
=&2\displaystyle\int_{\Omega}u\nabla v\cdot\nabla u-2\int_{\Omega}(u-1)uz.
\end{array}
\end{equation*}	
We thereupon make use of \eqref{3.6}  and Lemma \ref{lemma3.3} along with the Young inequality to  get
\begin{equation}\label{5.14}
\begin{array}{rl}
&\displaystyle\frac{d}{dt}\int_{\Omega}(u-1)^2+\int_{\Omega}|\nabla u|^2+2\mu a_* \int_{\Omega}(u-1)^2\\
\leq &\displaystyle \int_{\Omega}u^2|\nabla v|^2+2\int_{\Omega}uz\\
\leq & c_1\displaystyle \int_{\Omega}|\nabla v|^2+c_1\int_{\Omega}z
\end{array}
\end{equation}	
with some $ c_1>0$.

Thanks to the outcomes of Lemma \ref{lemma5.3} and Lemma \ref{lemma3.4},  \eqref{5.14} readily leads to
\begin{equation}\label{5.15}
\int_{\Omega}(u-1)^2\leq c_2 e^{-2\eta t}
\end{equation}	
with $\eta:=\min\{\mu a_*, a_*,\frac{\delta}2\}$ and  $c_2>0$ for all $t>0$.

According to known smoothing estimates for the Neumann heat semigroup on $\Omega\subset \mathbb{R}^3$ \cite{WinklerJDE},
 there exist $c_3=c_3(p,q)>0$, $c_4=c_4(p,q)>0$ fulfilling
\begin{equation}\label{5.16}
\left \| e^{\sigma \Delta } \varphi\right \|_{L^{p}(\Omega)}
\leq c_3 \sigma^{-\frac32(\frac 1q-\frac 1p)}
 \| \varphi  \|_{L^q(\Omega)}
\end{equation}
for each $\varphi\in  C^0(\Omega)$, and
 for all $\varphi\in \left ( L^{q}\left ( \Omega  \right ) \right )^{3}$,
\begin{equation}\label{5.17}
\left \| e^{\sigma \Delta }\nabla\cdot \varphi\right \|_{L^{p}\left ( \Omega  \right )}\leq c_{4} ( 1+\sigma^{-\frac{1}{2}-
\frac32( \frac{1}{q}-\frac{1}{p}) } )e^{-\lambda _{1}\sigma}\left \|  \varphi \right \|_{L^{q}\left ( \Omega  \right )}
\end{equation}
with $ \lambda _{1}> 0$  the first nonzero eigenvalue of $-\Delta$ in $\Omega $ under the Neumann boundary condition.

Applying a variation-of-constants representation of $u$ related to the the first equation in \eqref{1.3} and  utilizing \eqref{5.16} and \eqref{5.17}, we infer that
\begin{equation}\label{5.18}
\begin{array}{rl}
&\|(u-1)(\cdot,t)\|_{L^p(\Omega)}\\
\leq &
\|e^{t(\Delta-\eta )}(u_0-1)\|_{L^p(\Omega)}+
\displaystyle \int^{t}_0\|e^{(t-s)(\Delta-\eta)} \nabla\cdot(u\nabla v) \|_{L^p(\Omega)}ds\\
&+
\displaystyle\int^{t}_0\|e^{(t-s)(\Delta-\eta)}((\mu u-\eta)(1-u)- uz)\|_{L^p(\Omega)}ds	\\
 \leq &c_3(p) \displaystyle e^{-\eta t}\|u_0-1\|_{L^{p}(\Omega)}+
  c_4(p)
\int^{t}_0(1+(t-s)^{-\frac 54+\frac 3{2p}})e^{-(\eta+\lambda_1)(t-s)}\|\nabla v(\cdot,s)\|_{L^{2}(\Omega)}ds\\
&
+c_5(p)\displaystyle\int^{t}_0(1+(t-s)^{-\frac 34+\frac 3{2p}})e^{-\eta(t-s)}  \|(u-1)(\cdot,s)\|_{L^2(\Omega)}ds\\
&
+c_5(p)\displaystyle\int^{t}_0(1+(t-s)^{-\frac 34+\frac 3{2p}})e^{-\eta(t-s)}  \|z(\cdot,s)\|_{L^2(\Omega)}ds
\end{array}
\end{equation}
for  some $c_5(p)>0$.
Therefore by \eqref{5.15}, \eqref{5.10}, \eqref{3.16}  and  the fact that for
 $\alpha\in(0,1) $  $\gamma_1$ and  $\delta_1$ positive constants with $ \gamma_1 \neq \delta_1$, there exists $c_6> 0$ such that
$$\int_{0}^{t} ( 1+( t-s  ) ^{-\alpha})e^{-\gamma_1 s}e^{-\delta_1 ( t-s )}ds
\leq c_{6} e^{-min\left \{ \gamma_1 ,\delta_1  \right \}t},
$$
 \eqref{5.18} readily yields  \eqref{5.13} with $\gamma=\eta$ and some $C(p)>0$.
\end{proof}

It is noted that due to the fact that the integrability exponent in \eqref{5.10} does not exceed the considered spatial dimension $N=3$,
the uniform H\"{o}lder bounds for $u$ seems to be unavailable so far,
 though $\nabla v\in L^\infty_{loc}((0,\infty),L^5(\Omega))$ achieved in Lemma \ref{lemma4.2}.
 On the other hand, according to the  extensibility
criteria of the classical solution to \eqref{2.1}, we need to establish the global boundedness of $\|\nabla v(\cdot,t)\|_{L^5(\Omega)}$.
To this end, we first turn to
  make sure that  $\int_{\Omega}|\nabla v|^4$   decays exponentially and inter alia $a,b$ enjoy some higher regularity,
   which results from a series of testing procedures.



\begin{lemma}\label{lemma5.5}
Let the hypothesis in Theorem \ref{T1.1} hold. Then there exist $\alpha>0$ and  $C>0$
such that for all $t>0$,	
\begin{equation}\label{5.19}
\int_{\Omega}(|\nabla v(\cdot,t)|^4+|\nabla a(\cdot,t)|^2+|\nabla b(\cdot,t)|^2)\leq C e^{-\alpha t}
\end{equation}
as well as
\begin{equation}\label{5.20}
\int^\infty_0\int_{\Omega}(|\Delta a|^2+|\Delta b|^2)<\infty.
\end{equation}
\end{lemma}
\begin{proof}
 Testing the identity
\begin{equation*}
a_t= \triangle a +  \nabla v \cdot \nabla a + f(x,t), \quad x\in\Omega, \quad t > 0
\end{equation*}	
with $
f(x,t)=\mu a(1-u)-\displaystyle a z+\ a( u+ w)v$ by $-\triangle a$, and using Young's inequality, we get
\begin{equation}\label{5.21}
\begin{array}{rl}
\displaystyle \frac d {dt}\int_\Omega |\nabla a|^2 +2\displaystyle \int_\Omega|\triangle a|^2 =&-2\displaystyle \int_\Omega
 (\nabla a\cdot \nabla v)\triangle a-
2\int_\Omega f\triangle a\\
\leq&\displaystyle \int_\Omega|\triangle a|^2 +2\int_\Omega  |\nabla a|^2 |\nabla v|^2+ 2\int_\Omega  |f|^2.
\end{array}
\end{equation}
Proceeding as in the proof of \eqref{4.5}, we can find $c_1>0, c_2>0$ such that
 \begin{equation}\label{5.22}
\begin{array}{rl}
&\displaystyle \frac d {dt}\|\nabla a\|^2_ {L^2(\Omega)} + c_1\|\nabla a\|^2_ {L^2(\Omega)}+
c_1 \|\triangle a\|^2_ {L^2(\Omega)} \\[3mm]
\leq&    c_2  \|\nabla v\|^4_ {L^4(\Omega)}
+c_2\|f\|_{L^2(\Omega)}^2.\\[2mm]
\end{array}
\end{equation}
Likely, we also have
\begin{equation}\label{5.23}
\begin{array}{rl}
&\displaystyle \frac d {dt}\|\nabla b\|^2_ {L^2(\Omega)} + c_3\|\nabla b\|^2_ {L^2(\Omega)}+
c_3 \|\triangle b\|^2_ {L^2(\Omega)} \\[3mm]
\leq&    c_4  \|\nabla v\|^4_ {L^4(\Omega)}
+c_4\|g\|_{L^2(\Omega)}^2\\[2mm]
\end{array}
\end{equation}
for some $c_3>0,c_4>0$, where  $
g(x,t)=- b+\displaystyle ue^vz
+b(u+ w)v$.

To appropriately compensate the first summand on right-hand side of \eqref{5.22} and  \eqref{5.23},  we use
the third equation in \eqref{2.1} to see that
\begin{equation}\label{5.24}
\begin{array}{rl}
&\displaystyle \frac 14 \frac d {dt}\int_\Omega |\nabla v|^4\\
  =& -\displaystyle  \int_\Omega|\nabla v|^2 \nabla v\cdot\nabla v_t\\[2mm]
  =&- \displaystyle \int_\Omega   a(v+1)e^{ v}|\nabla v|^4-
\displaystyle \int_\Omega   ve^{v}|\nabla v|^2 \nabla v\cdot\nabla a\\[3mm]
&-\displaystyle \int_\Omega   b(v+1)e^{v}|\nabla v|^4-
 \displaystyle \int_\Omega   ve^{ v}|\nabla v|^2 \nabla v\cdot\nabla b.
\end{array}
\end{equation}
Here recalling the uniform positivity of $a$ stated in Lemma \ref{lemma3.2}, we can pick $c_5>0$ fulfilling
$$
 \displaystyle \int_\Omega   a(v+1)e^{ v}|\nabla v|^4
\geq c_5 \displaystyle \int_\Omega  |\nabla v|^4
$$
and thus  infer by the Young inequality and  Lemma \ref{lemma3.4} that for all $t>0$
\begin{equation}\label{5.25}
\begin{array}{rl}
&\displaystyle  \frac d {dt}\int_\Omega |\nabla v|^4 +c_5\int_\Omega |\nabla v|^4
\\
  \leq &\displaystyle\frac 2{c_5} e^{-\delta t}\displaystyle \int_\Omega  (|\nabla a|^4+ |\nabla b|^4)\\
  \leq &\displaystyle c_6 e^{-\delta t}\displaystyle \int_\Omega  (|\triangle a|^2+ |\triangle b|^2)
\end{array}
\end{equation}
with constant $c_6>0$, where we have used the  Gagliardo--Nirenberg type inequality \eqref{4.4} in the last inequality.

Now   by  the appropriate  linear combination of \eqref{5.22}, \eqref{5.23} and \eqref{5.25}, we can see that
there exists $t_1>1$ suitably large such that
for all $t>t_1$, \begin{equation}\label{5.26}
\begin{array}{ll}
&\displaystyle\frac d {dt}\left(\|\nabla a\|^2_ {L^2(\Omega)} + \|\nabla b\|^2_ {L^2(\Omega)}+c_7\|\nabla v\|_ {L^4(\Omega)}^4\right)+\displaystyle c_8(\|\triangle a\|^2_ {L^2(\Omega)}+ \|\triangle b\|^2_ {L^2(\Omega)})\\[3mm]
+& c_8 \left(\|\nabla a\|^2_ {L^2(\Omega)} + \|\nabla b\|^2_ {L^2(\Omega)}+c_7\|\nabla v\|_ {L^4(\Omega)}^4\right)\\
\leq&  \displaystyle \frac 1{c_8}(\|f\|_{L^2(\Omega)}^2+\|g\|_{L^2(\Omega)}^2).
\end{array}
\end{equation}
 Due to the global boundedness of $a,b,z,v$ achieved in the previous Lemmas, we have
 $$
 |f(x,t)|^2+|g(x,t)|^2\leq c_9(|u(x,t)-1|^2+|b(x,t)|^2+ |z(x,t)|^2+|v(x,t)|^2)
 $$
  with  $c_9>0$,  and thereby there exist $\eta_1>0$ and  $c_{10}>0$
such that
\begin{equation}\label{5.27}
\int_{\Omega}|f(\cdot,t)|^2+ |g(\cdot,t)|^2 \leq c_{10} e^{-\eta_1 t}~~\hbox{for all $t>t_1,$}	
\end{equation}
thanks to   Lemma \ref{lemma5.4}, Lemma \ref{lemma3.2}, Lemma \ref{lemma3.3} and Lemma \ref{lemma3.4}.
Therefore from \eqref{5.26} and \eqref{5.27}, it follows that function
$y(t):=\|\nabla a\|^2_ {L^2(\Omega)} + \|\nabla b\|^2_ {L^2(\Omega)}+c_7\|\nabla v\|_ {L^4(\Omega)}^4
$  satisfies
\begin{equation}\label{5.28}
y'(t)+c_8 y(t)+ \displaystyle c_8(\|\triangle a\|^2_ {L^2(\Omega)}+ \|\triangle b\|^2_ {L^2(\Omega)})\leq \frac{c_{10}}{c_8} e^{-\eta_1 t},
\end{equation}
 and thereby \eqref{5.19} is readily valid with $\alpha=\min\{c_8,\eta_1 \}$. Thereafter  \eqref{5.20} results from an integration of
\eqref{5.28}.
 \end{proof}

Now we can turn the information contained in \eqref{5.19} and  \eqref{5.20} into  the global boundedness of $\|\nabla v(\cdot,t)\|_{L^5(\Omega)}$
 in the three-dimensional framework beyond that in \eqref{4.2}.
\begin{lemma}\label{lemma5.6}
Let the hypothesis in Theorem \ref{T1.1} hold. Then there exists  $C>0$
such that for all $t>0$,	
\begin{equation}\label{5.29}
\int_{\Omega}|\nabla v(\cdot,t)|^5\leq C.
\end{equation}
\end{lemma}
\begin{proof}
Proceeding as in the proof of \eqref{4.12}, we have
\begin{equation}\label{5.30}
\|\nabla v(\cdot,t)\|_{L^5(\Omega)}\leq \|\nabla v_0\|_{L^5(\Omega)}+e\int^{t}_0
(\|\nabla a(\cdot,s)\|_{L^5(\Omega)}+\|\nabla b(\cdot,s)
\|_{L^5(\Omega)})ds.
\end{equation}
Note that there exist $c_1>0,c_2>0$ such that  for all $\varphi\in C^2(\overline{\Omega}), \frac {\partial \varphi}{\partial \nu} =0$, $\|\varphi-\frac1{|\Omega|}\int_\Omega\varphi\|_{W^{2,2}(\Omega)}\leq c_1 \|\Delta \varphi\|_{L^2(\Omega)}$ and
$\|\varphi-\frac1{|\Omega|}\int_\Omega\varphi\|_{L^{2}(\Omega)}\leq c_2 \|\nabla \varphi\|_{L^2(\Omega)}$, so we have
 $$\|\nabla a(\cdot,s)\|_{L^5(\Omega)}\leq c_3 \|\triangle a(\cdot,s)\|^{\frac 9{10}}_{L^2(\Omega)}
\|\nabla a(\cdot,s)\|_{L^2(\Omega)}^{\frac 1{10}}
$$
as well as
$$\|\nabla b(\cdot,s)\|_{L^5(\Omega)}\leq c_3\|\triangle b(\cdot,s)\|^{\frac 9{10}}_{L^2(\Omega)}
\|\nabla b(\cdot,s)\|_{L^2(\Omega)}^{\frac 1{10}}
$$
for $c_3>0$,
and thereby
\begin{equation}\label{5.31}
\begin{array}{rl}
\displaystyle\int^{t}_0
\|\nabla a(\cdot,s)
\|_{L^5(\Omega)}ds&\leq c_3 \displaystyle\int^t_0 \|\Delta a(\cdot,s)\|_{L^2(\Omega)}^{\frac 9{10}} \|\nabla a(\cdot,s)\|_{L^2(\Omega)}^{\frac 1{10}}ds \\[3mm]
&\leq  c_3\left\{\displaystyle\int^t_0 \|\Delta a(\cdot,s)\|_{L^2(\Omega)}^2ds \right\}^{\frac 9{20} }
 \left\{\displaystyle\int^t_0
 \|\nabla a(\cdot,s)\|_{L^2(\Omega)}^{\frac 2{11}}ds \right\}^{\frac {11}{20} }\\
  &\leq  c_4\left\{\displaystyle\int^\infty _0 \|\Delta a(\cdot,s)\|_{L^2(\Omega)}^2 ds\right\}^{\frac 9{20} }
 \left\{\displaystyle\int^t_0
 e^{-\frac \alpha {11}s}ds \right\}^{\frac {11}{20} }\\
  &< c_5
\end{array}
\end{equation}
as well as
\begin{equation}\label{5.32}
\displaystyle\int^{t}_0
\|\nabla b(\cdot,s)
\|_{L^5(\Omega)}ds
< c_6
\end{equation}
for some $c_5>0,c_6>0$, where we have used \eqref{5.20} and \eqref{5.19}. Hence \eqref{5.29} results
readily from \eqref{5.30}--\eqref{5.32}.
 \end{proof}

At this position, thanks to the known smoothing estimates for the Neumann heat semigroup again, we can readily turn the information
contained in Lemma \ref{lemma5.5} into
 the exponential decay property of $u-1$  with respect to  $L^\infty(\Omega)$-norm.
\begin{lemma}\label{lemma5.7}
Let the assumptions in Theorem \ref{T1.1} hold. Then there exist $\vartheta>0$ and  $C>0$ fulfilling
\begin{equation}\label{5.33}
\|u(\cdot,t)-1\|_{L^\infty(\Omega)}\leq C e^{-\vartheta t}.
\end{equation}	
\end{lemma}
\begin{proof}
 Since the proof is similar to that of  Lemma \ref{lemma5.4}, we only give a short proof of  \eqref{5.33}.
In view to known smoothing estimates for the Neumann heat semigroup on $\Omega\subset \mathbb{R}^2$ (\cite{WinklerJDE}),
 there exist $c_1>0$, $c_2>0$ fulfilling
\begin{equation}\label{5.34}
\left \| e^{\sigma \Delta } \varphi\right \|_{L^{\infty}(\Omega)}
\leq c_1 \sigma^{-\frac 34}
 \| \varphi  \|_{L^2(\Omega)}
\end{equation}
for each $\varphi\in  C^0(\Omega)$, and
 for all $\varphi\in \left ( L^{4}\left ( \Omega  \right ) \right )^{3}$,
\begin{equation}\label{5.35}
\left \| e^{\sigma \Delta }\nabla\cdot \varphi\right \|_{L^{\infty}\left ( \Omega  \right )}\leq c_2 ( 1+\sigma^{-\frac78} )e^{-\lambda_{1}\sigma}\left \|  \varphi \right \|_{L^{4}\left ( \Omega  \right )}
\end{equation}
with $ \lambda _{1}> 0$  the first nonzero eigenvalue of $-\Delta$ in $\Omega $ under the Neumann boundary condition.

According to the variation-of-constants representation of $u$ related to the the first equation in \eqref{1.3}, we utilize \eqref{5.34} and \eqref{5.35} to infer that
\begin{equation}\label{5.36}
\begin{array}{rl}
&\|(u-1)(\cdot,t)\|_{L^\infty(\Omega)}\\
\leq &
\|e^{t(\Delta-1 )}(u_0-1)\|_{L^\infty(\Omega)}+
\displaystyle \int^{t}_0\|e^{(t-s)(\Delta-1)} \nabla\cdot(u\nabla v)(\cdot,s) \|_{L^\infty(\Omega)}ds\\
&+
\displaystyle\int^{t}_0\|e^{(t-s)(\Delta-1)}((\mu u-1)(1-u)-uz)(\cdot,s) \|_{L^\infty(\Omega)}ds	\\
 \leq &
 \displaystyle e^{-t}\|u_0-1\|_{L^{\infty}(\Omega)}+
  c_3
\int^{t}_0(1+(t-s)^{-\frac78})e^{-(1+\lambda_1)(t-s)}\|\nabla v(\cdot,s)\|_{L^{4}(\Omega)}ds\\
&
+c_3\displaystyle\int^{t}_0(1+(t-s)^{-\frac34})e^{-(t-s)}
 (\|(u-1)(\cdot,s)\|_{L^2(\Omega)}
+ \|z(\cdot,s)\|_{L^2(\Omega)})ds
\end{array}
\end{equation}
with some $c_3>0$.
This  readily  establishes \eqref{5.33} with appropriate $\vartheta>0$ in view of  \eqref{5.19}, \eqref{5.13} and \eqref{3.16}.
\end{proof}

Thereby  our main result has essentially been proved already.

{\bf Proof of Theorem 1.1.}  \rm The statement on global boundedness of classical solutions
has been asserted by  Lemma \ref{lemma3.2}--Lemma \ref{lemma3.4} and Lemma \ref{lemma5.6}.  The convergence properties
in \eqref{1.8}--\eqref{1.11} are
precisely established by  Lemma \ref{lemma3.2}--Lemma \ref{lemma3.4} and Lemma \ref{lemma5.7}, respectively.

\section{Acknowledgments}
This work is partially supported by NNSFC (No.12071030).

\end{document}